\documentclass[aps,prl,final,twocolumn,showpacs]{revtex4}

\usepackage{amsmath,amssymb,amsthm}
\usepackage{graphicx}
\usepackage{epsf,url}

\newtheorem{theorem}{Theorem}
\newtheorem{prop}{Proposition}
\newtheorem{lemma}{Lemma}

\begin{document}

\title{Dynamical properties of $k$-free lattice points}

\author{Christian Huck and Michael Baake}
\affiliation{Fakult\"{a}t f\"{u}r Mathematik, 
Universit\"{a}t Bielefeld,
Postfach 100131, 33501 Bielefeld, Germany}

\begin{abstract}
 We revisit the visible points of a lattice in Euclidean $n$-space together with their
generalisations, the $k$th-power-free points of a lattice, and study the corresponding
dynamical system that arises via the closure of the lattice
translation orbit. Our analysis extends
previous results obtained by Sarnak and by Cellarosi and
Sinai for the special case of square-free integers and sheds new light
on previous joint work with Peter Pleasants.
 \end{abstract}


\maketitle

\section{Introduction}

In \cite{BMP}, the diffraction properties of the visible points of
$\mathbb Z^2$ and the
$k$th-power-free numbers were studied. It was shown that these sets
have positive, pure-point, translation-bounded \emph{diffraction spectra} with countable,
dense support. This is of interest because these sets
fail to be Delone sets: they are uniformly discrete (subsets of lattices, in fact)
but not relatively dense. The lack of relative denseness means that these
sets have arbitrarily large `holes'. In \cite{PH}, it was shown that
the above results remain true for the larger class of $k$th-power-free
(or $k$-free for short) 
points of arbitrary lattices in $n$-space. Furthermore, it was shown
there that these sets have positive \emph{patch counting entropy} but
zero \emph{measure-theoretical entropy} with respect to a measure that is defined in terms of the
`tied' frequencies of patches in space. 

Recent independent results by Sarnak~\cite{Sarnak} and by Cellarosi and
Sinai~\cite{CS} on the natural dynamical system associated with the square-free (resp.\
$k$th-power-free) integers (in particular on the ergodicity of the
above frequency measure and the dynamical spectrum, but also on the topological dynamics) go beyond what was covered
in~\cite{PH}. The aim of this short note is to generalise these
results to the setting of $k$-free lattice points. 

\section{$k$-free points}

The \emph{$k$-free points} $V=V(\varLambda,k)$ of a lattice
$\varLambda\subset\mathbb R^n$ are the points with the property that the
greatest common divisor of their 
coordinates in any lattice basis is not divisible by any non-trivial
$k$th power of an integer.  Without restriction, we shall assume that
$\varLambda$ is
unimodular, i.e.\ $|\det(\varLambda)|=1$. One can see that $V$ is
\emph{non-periodic}, i.e.\ $V$ has no non-zero translational symmetries. As particular cases, we have
the visible points (with respect to the origin $0$) of $\varLambda$ (with $n\ge2$ and $k=1$) and the
$k$-free integers (with $\varLambda=\mathbb Z$), both treated in
\cite{BMP} and \cite{BG}. We exclude the trivial case $n=k=1$, where $V$ consists
of just the two points of $\varLambda$ closest to $0$ on either side.

\begin{center}
\begin{figure}
\centerline{\epsfysize=0.48\textwidth\epsfbox{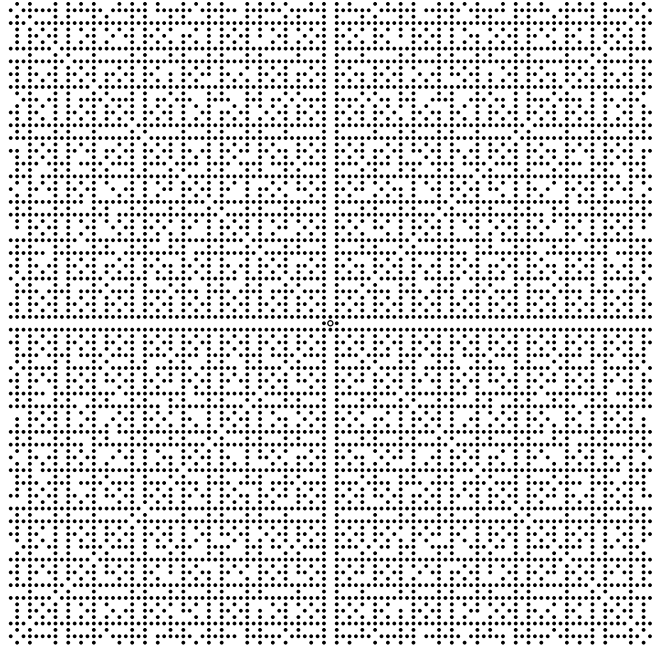}}
\caption{A central patch of the visible points of the square lattice
  $\mathbb Z^2$. Note the invariance with respect to $\operatorname{GL}(2,\mathbb Z)$.}
\label{fig:vis}
\end{figure}
\end{center}

Let $v_n=\operatorname{vol}(B_1(0))$, so that $v_nR^n$ is the volume
of the open ball $B_R(0)$ of radius $R$ about $0$. If $Y\subset\varLambda$, its `tied' \emph{density} $\operatorname{dens}(Y)$ is defined by
$$
\operatorname{dens}(Y):=\lim_{R\to\infty}\frac{|Y\cap B_R(0)|}{v_nR^n},
$$
when the limit exists. The following result is well known.

\begin{theorem}{\rm \cite[Cor.~1]{PH}}
One has $\operatorname{dens}(V)=1/\zeta(nk)$,
where $\zeta$ denotes
Riemann's $\zeta$-function.
\qed
\end{theorem}

An application of the Chinese Remainder Theorem immediately gives the
following result on the occurrence of `holes' in $V$.

\begin{prop}{\rm \cite[Prop. 1]{PH}}\label{holes}
$V$ is uniformly discrete, but has arbitrarily large holes. Moreover,
for any $r>0$, the set of centres of holes in $V$ of inradius at least $r$
contains a coset of $m^k\varLambda$ in $\varLambda$ 
for some $m\in\mathbb N$.
\qed
\end{prop}

Given a radius $\rho>0$ and a point $t\in\Lambda$,
the \emph{$\rho$-patch} of $V$ at $t$ is
\[(V-t)\cap B_\rho(0),\]
the translation to the origin of the part of $V$ within a distance
$\rho$ of $t$. We denote by $\mathcal A(\rho)$ the (finite) set of all
$\rho$-patches of $V$, and by $N(\rho)=|\mathcal A(\rho)|$ the number of
distinct $\rho$-patches of $V$. In view of the binary configuration
space interpretation, and following~\cite{PH}, the \emph{patch counting entropy} of
$V$ is defined as
$$
h_{\rm pc}(V):=\lim_{\rho\to\infty}\frac{\log_2N(\rho)}{v_n\rho^n}.
$$
It can be shown by a classic subadditivity argument that this limit exists.

Following~\cite{BMP,PH}, the `tied' \emph{frequency} $\nu(\mathcal P)$
of a $\rho$-patch $\mathcal P$ of $V$ is defined by
\begin{equation}\label{freqdef}
\nu(\mathcal
P):=\operatorname{dens}\big(\{t\in\varLambda\,\mid\,(V-t)\cap B_\rho(0)=\mathcal P\}\big),
\end{equation}
which can indeed be seen to exist. Moreover, one has

\begin{theorem}{\rm \cite[Thms.~1 and~2]{PH}}\label{freq}
Any $\rho$-patch $\mathcal P$ of $V$ occurs with positive frequency, given by
\[\nu(\mathcal P)=\sum_{\mathcal F\subset (B_{\rho}(0)\cap\varLambda)\setminus \mathcal P}(-1)^{|\mathcal F|}
\prod_p\left(1-\frac{|(\mathcal P\cup\mathcal
    F)/p^k\varLambda|}{p^{nk}}\right),\]
where $p$ runs through all primes.
\qed
\end{theorem}

\section{Diffraction}

Recall that the \emph{dual}\/ or \emph{reciprocal 
lattice}\/ $\varLambda^*$ of $\varLambda$ is
\[ 
\varLambda^*:=\{y \in\mathbb{R}^n\,\mid\, y\cdot x\in\mathbb Z
\mbox{ for all } x\in\varLambda\}.
\]
Further, the
\emph{denominator} of a point $y$ in the $\mathbb Q$-span $\mathbb
Q\varLambda^*$ of $\varLambda^*$ is defined as
$$
\operatorname{den}(y):=\min\{m\in\mathbb N\,\mid\,m y\in\varLambda^*\}.
$$

\begin{theorem}{\rm \cite[Thms.~3 and 5]{BMP} \cite[Thm.~8]{PH} \cite{BG}}\label{thdiff}
The natural diffraction measure $\widehat{\gamma}$ of the autocorrelation
$\gamma$ of\/ $V$ exists and is a positive, 
translation-bounded, pure-point measure
which is concentrated on the set of points in $\mathbb Q\varLambda^*$
with $(k+1)$-free denominator, the Fourier--Bohr spectrum of $\gamma$, 
and whose intensity is 
\[
\Bigg(\frac{1}{\zeta(nk)}\prod_{p\mid q}\frac{1}{p^{nk}-1}\Bigg)^2
\]
at any point with such
a denominator $q$.
\qed
\end{theorem}

\begin{center}
\begin{figure}
\centerline{\epsfysize=0.48\textwidth\epsfbox{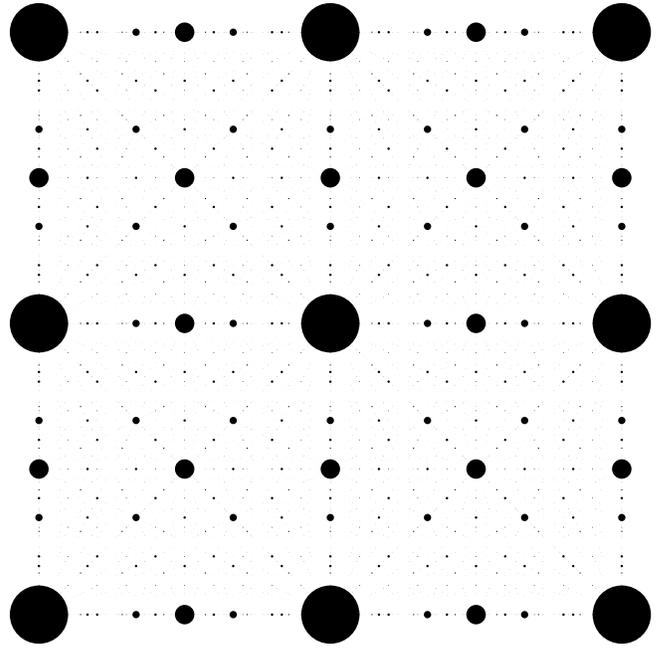}}
\caption{Diffraction $\widehat{\gamma}$ of the visible points of $\mathbb Z^2$. Shown are
the intensities with $I(y)/I(0)\ge 10^{-6}$ and $y\in [0,2]^2$. Its lattice of periods is $\mathbb Z^2$, and $\widehat{\gamma}$ turns out to
be $\operatorname{GL}(2,\mathbb Z)$-invariant.}
\label{fig:diff}
\end{figure}
\end{center}

\section{The hull of $V$}

Endowing the power set $\{0,1\}^{\varLambda}$ of the lattice $\varLambda$ with the
product topology of the discrete topology on $\{0,1\}$, it becomes a
compact topological space (by Tychonov's theorem). This
topology is in fact generated  by the metric $\rm d$ defined by
$$
{\rm d}(X,Y):=\min\Big\{1,\inf_{\epsilon >0}\{X\cap
  B_{1/\epsilon}(0)=Y\cap
  B_{1/\epsilon}(0)\}\Big\}
$$
for subsets $X,Y$ of $\varLambda$; cf.~\cite{Sol}. Then, $(\{0,1\}^{\varLambda},\varLambda)$ is
a \emph{topological dynamical system}, i.e.\ the natural translational
action of the group $\varLambda$
on $\{0,1\}^{\varLambda}$ is continuous. 

Let $X$ now be a subset of
$\varLambda$. The closure 
$$\mathbb X(X):=\overline{\{t+X\,\mid\,t\in \varLambda\}}$$ of the set of 
lattice translations $t+X$  of $X$ in
$\{0,1\}^{\varLambda}$ is the \emph{$($discrete\/$)$ hull} of $X$ and gives rise to the topological dynamical system
$(\mathbb X(X),\varLambda)$, i.e.\ $\mathbb X(X)$ is a compact topological space on which
the action of $\varLambda$ is continuous. 

By construction of the hull, Proposition~\ref{holes} implies

\begin{lemma}\label{holes2}
For any $r>0$ and any element $X\in\mathbb X(V)$, the set of centres of holes in $X$ of inradius at least $r$
contains a coset of $m^k\varLambda$ in $\varLambda$ 
for some $m\in\mathbb N$.
\qed
\end{lemma}

For a $\rho$-patch $\mathcal P$ of $V$, denote by $C_\mathcal{P}$ the set of elements of $\mathbb X(V)$ whose $\rho$-patch at
$0$ is $\mathcal P$, the so-called \emph{cylinder set} defined by the
$\rho$-patch $\mathcal P$. Note that these cylinder sets form a
basis of the topology of $\mathbb X(V)$.

It is clear from the existence of holes of unbounded inradius in $V$ that
$\mathbb X(V)$ contains the empty set  (the configuration of $0$ on
every lattice point). Denote by
$\mathbb A$
the set of \emph{admissible} subsets $A$
of $\varLambda$, i.e.\ subsets $A$ of $\varLambda$ having the property that, 
for every prime $p$, $A$ does \emph{not} contain a full set of
representatives modulo $p^k\varLambda$. In other words, $A$ is
admissible if and only if 
$|A/p^k\varLambda|<p^{nk}$ for any prime $p$, where
$A/p^k\varLambda$ denotes the set of cosets of $p^k\varLambda$ in $\varLambda$ that are
represented in $A$. Since $V\in\mathbb A$ (otherwise some point of $V$ is in $p^k\varLambda$
for some prime $p$, a contradiction) and since $\mathbb A$ is a $\varLambda$-invariant and
closed subset of $\{0,1\}^{\varLambda}$, it is clear that $\mathbb X(V)$
is a subset of $\mathbb A$. By~\cite[Thm.~2]{PH}, the other inclusion is also
true. One thus obtains the
following characterisation of the hull of $V$, which was first shown by Sarnak~\cite{Sarnak} for the
special case of the square-free integers.

\begin{theorem}{\rm \cite[Thm.~6]{PH}}
One has $\mathbb X(V)=\mathbb A$.
\qed
\end{theorem} 

In particular, $\mathbb X(V)$ contains
\emph{all} subsets of $V$ (and their translates). In other words, $V$
is an interpolating set for $\mathbb X(V)$ in the sense
of~\cite{W}, i.e. $$\mathbb X(V)|^{}_V\,\,:=\{X\cap V\,\mid\,
X\in\mathbb X(V)\}=\{0,1\}^V.$$
It follows that $V$ has patch counting entropy at least
$\operatorname{dens}(V)=1/\zeta(nk)$. In fact, one has more.

\begin{theorem}{\rm \cite[Thm.~3]{PH}~\cite[Thm.~1]{BLR}}\label{hpc}
One has $h_{\rm pc}(V)=1/\zeta(nk)$. Moreover, $h_{\rm pc}(V)$ coincides with the
topological entropy of the dynamical system $(\mathbb X(V),\varLambda)$.
\qed
\end{theorem} 
 
\section{Topological dynamics}  

By construction, $(\mathbb X(V),\varLambda)$ is topologically
transitive~\cite{A,G,W}, as it is the orbit closure of one of its
elements (namely $V$). Equivalently, for any two non-empty open subsets $U$
and $W$ of $\mathbb X(V)$, there is an element $t\in\varLambda$ such
that
$$
U\cap (W+t)\neq\varnothing.
$$
In accordance with Sarnak's findings~\cite{Sarnak} for
square-free integers, one has the following results.

\begin{theorem}\label{c1}
The topological dynamical system $(\mathbb X(V),\varLambda)$ has the following properties.
\begin{itemize}
\item[\rm (a)]
$(\mathbb X(V),\varLambda)$ is topologically ergodic with positive topological 
entropy equal to $1/\zeta(nk)$.
\item[\rm (b)]
$(\mathbb X(V),\varLambda)$ is proximal $($i.e., for any $X,Y\in\mathbb
X(V)$ one has $\inf_{t\in\varLambda}{\rm d}(X+t,Y+t)=0$$)$ and $\{\varnothing\}$ is
the unique $\varLambda$-minimal subset of $\mathbb
X(V)$.
\item[\rm (c)]
$(\mathbb X(V),\varLambda)$ has no non-trivial topological Kronecker
factor $($i.e., minimal equicontinuous factor\/$)$. In particular, $(\mathbb
X(V),\varLambda)$ has trivial topological point spectrum.
\item[\rm (d)]
$(\mathbb X(V),\varLambda)$ has a non-trivial joining with the Kronecker system $K=(G,\varLambda)$, where $G$ is the compact
Abelian group $\prod_p (\varLambda/p^k\varLambda)$ and $\varLambda$ acts on $G$
via addition on the diagonal, $g\mapsto
g+(\bar{x},\bar{x},\dots)$, with $g\in G$ and $x\in\varLambda$. In
particular, $(\mathbb X(V),\varLambda)$ fails to be topologically weakly mixing. 
\end{itemize}
\end{theorem}
\begin{proof}
The positivity of the topological entropy follows from
Theorem~\ref{hpc} since $1/\zeta(nk)>0$. For the
topological ergodicity~\cite{A}, one has to show that, for any two non-empty
open subsets $U$ and $W$ of $\mathbb X(V)$, one has 
\begin{equation}\label{tergod}
\limsup_{R\to\infty}\frac{\sum_{t\in\varLambda\cap B_R(0)}\theta\big(U\cap
  (W+t)\big)}{v_nR^n}>0,
\end{equation}
where $\theta(\varnothing)=0$ and $\theta(A)=1$ for non-empty subsets
$A$ of $\mathbb X(V)$. It certainly suffices to verify~\eqref{tergod}
for cylinder sets. To this end, let $\mathcal P$ and $\mathcal Q$ be
patches of $V$. Then, a suitable translate $V+s$ is an element of $C_\mathcal P$. Since
\begin{eqnarray*}
&&\hspace{-2em}\limsup_{R\to\infty}\frac{\sum_{t\in\varLambda\cap
    B_R(0)}\theta\big(C_\mathcal P\cap
  (C_\mathcal Q+t)\big)}{v_nR^n}\\
&\ge&\limsup_{R\to\infty}\frac{\sum_{t\in\varLambda\cap
    B_R(0)}\theta\big(\{V+s\}\cap
  (C_\mathcal Q+t)\big)}{v_nR^n}\\
&=&\limsup_{R\to\infty}\frac{\sum_{t\in\varLambda\cap
    B_R(0)}\theta\big(\{V\}\cap
  (C_\mathcal Q+t)\big)}{v_nR^n}\,=\,\nu(\mathcal Q),
\end{eqnarray*}
the assertion follows from Theorem~\ref{freq}. This proves (a).

For part (b), one can easily derive from Lemma~\ref{holes2} that, for any $\rho>0$
and any two
elements $X,Y\in\mathbb
X(V)$, there is a translation $t\in\varLambda$ such that
$$(X+t)\cap B_\rho(0)=(Y+t)\cap B_\rho(0)=\varnothing,$$ i.e.\ both $X$ and $Y$
have the empty $\rho$-patch at $-t$. It follows that ${\rm d}(X+t,Y+t)\le
1/\rho$ and thus the proximality of the system follows. Similarly, the assertion on the unique $\varLambda$-minimal
subset $\{\varnothing\}$ follows from the fact that any element of $\mathbb
X(V)$ contains arbitrarily large `holes' and thus any non-empty
subsytem contains $\varnothing$. 

Since Kronecker systems are distal, the first assertion of part (c) is an immediate consequence of the
proximality of $(\mathbb X(V),\varLambda)$. Although this immediately
implies that $(\mathbb X(V),\varLambda)$ has trivial
topological point spectrum, we add the following independent argument. Let $f\!:\, \mathbb
X(V)\longrightarrow \mathbb C$ be a continuous eigenfunction, in
particular $f\not\equiv 0$.
Let $\lambda_t\in\mathbb C$ be the eigenvalue with respect to $t\in\varLambda$,
i.e.\ $f(X-t)=\lambda_t f(X)$ for any $X\in\mathbb X(V)$, in
particular
\begin{equation}\label{emptyset}
f(\varnothing)=\lambda_tf(\varnothing).
\end{equation} 
Since
$\varLambda$ acts by homeomorphisms on the compact space $\mathbb X(V)$ and since $(\mathbb
X(V),\varLambda)$ is topologically transitive, it is clear that
$|\lambda_t|=1$ and that $|f|$ is a non-zero constant. We shall now show
that even $\lambda_t=1$ for any $t$ and that $f$ itself is a non-zero constant. By
Lemma~\ref{holes2}, for any $X\in\mathbb X(V)$, one can choose a sequence $(t_n)_{n\in\mathbb N}$
in $\varLambda$ such that $\lim_{n\rightarrow
  \infty}(X-t_n)=\varnothing$. Since $f$ is continuous, we have 
\begin{equation}\label{emptyset2}f(\varnothing)=\lim_{n\rightarrow
    \infty}f(X-t_n)=\lim_{n\rightarrow \infty}\lambda_{t_n}f(X).
\end{equation}  
Assuming
that $f(\varnothing)=0$ thus implies $f\equiv 0$, a
contradiction. Hence $f(\varnothing)\neq 0$ and $\lambda_t=1$ for any
$t\in\varLambda$ by~\eqref{emptyset}. Further,  by~\eqref{emptyset2},
one has 
$f(X)=f(\varnothing)$ for any $X\in\mathbb X(V)$.    

For part (d), one can verify that a non-trivial joining~\cite{G} 
of $(\mathbb X(V),\varLambda)$ with the Kronecker system $K$ is given
by 
$$
W:=\bigcup_{X\in\mathbb X(V)}\Big(\{X\}\times \prod_p
(\varLambda\setminus X) /p^k\varLambda\Big).
$$
Since the Kronecker system $K$ is minimal and distal, a well known
disjointness theorem by Furstenberg~\cite[Thm. II.3]{F} implies that
$(\mathbb X(V),\varLambda)$ fails to be topologically weakly mixing.
\end{proof}

\section{Measure-theoretic dynamics}

The frequency function $\nu$ from~\eqref{freqdef}, regarded as a function on the
cylinder sets by setting $\nu(C_\mathcal
P):=\nu(\mathcal P)$, is finitely additive on the
cylinder sets with $$\nu(\mathbb X(V))=\sum_{\mathcal P\in\mathcal
  A(\rho)}\nu(C_{\mathcal P})=|\det(\varLambda)|=1.$$ Since the
family of cylinder sets is a (countable) semi-algebra that generates the
Borel $\sigma$-algebra on $\mathbb X(V)$ (i.e.\ the smallest 
$\sigma$-algebra on $\mathbb X(V)$ which contains the open subsets of
$\mathbb X(V)$), it extends uniquely to a probability measure on
$\mathbb X(V)$; cf.~\cite[\S 0.2]{Walters}. Moreover, this probability measure is
$\varLambda$-invariant by construction. For part (b) of the following
claim, note that, in the case of $V$,  
the Fourier--Bohr spectrum is itself a group and
compare~\cite[Prop. 17]{BLvE}. Turning to the measure-theoretic
dynamical system $(\mathbb X(V),\varLambda,\nu)$, one has

\begin{theorem}$(\mathbb X(V),\varLambda,\nu)$ has the following properties.
\begin{itemize}
\item[\rm (a)]
The $\varLambda$-orbit of $V$ in $\mathbb X(V)$ is $\nu$-equidistributed,
i.e., for any function $f\in C(\mathbb X(V))$, one has
\[
\lim_{R\to\infty}\frac{1}{v_nR^n}\sum_{x\in\varLambda\cap
  B_R(0)}f(V+x)=\int_{\mathbb X(V)}f(X)\,\,{\rm d}\nu(X).
\]
In other words, $V$ is $\nu$-generic.
\item[\rm (b)]
$(\mathbb X(V),\varLambda,\nu)$ is ergodic, deterministic
$($i.e., it is of zero measure entropy\/$)$ and has pure-point dynamical spectrum given by
the Fourier--Bohr spectrum of the autocorrelation $\gamma$, as
described in Theorem~$\ref{thdiff}$.
\item[\rm (c)]
The Kronecker system $K_{\nu}=(X_K,\varLambda,\nu)$, where $X_K$ is the
compact Abelian 
group $\prod_p (\varLambda/p^k\varLambda)$, $\varLambda$ acts on $X_K$ via
addition on the diagonal $($cf.\ Theorem\/~$\ref{c1}(\rm{d}))$ and $\nu$ is
Haar measure on $X_K$, is metrically 
isomorphic to $(\mathbb X(V),\varLambda,\nu)$.
\end{itemize}
\end{theorem}
\begin{proof}
For part (a), it suffices to show this for the characteristic
functions of cylinder sets of finite patches, as their span is dense
in $C(\mathbb X(V))$. But for such functions, the claim is clear as
the left hand side is the patch frequency as used for the definition
of the measure $\nu$.

For the ergodicity of $(\mathbb X(V),\varLambda,\nu)$, one has to show
that
$$
\lim_{R\rightarrow\infty}\frac{1}{v_nR^n}\sum_{x\in\varLambda\cap
  B_R(0)}\nu\big((C_\mathcal P+x)\cap C_\mathcal Q\big)=\nu(C_\mathcal P)\nu(C_\mathcal Q)
$$
for arbitrary cylinder sets $C_\mathcal P$ and $C_\mathcal Q$;
compare~\cite[Thm. 1.17]{Walters}. The latter in turn follows from a straightforward calculation using
Theorem~\ref{freq} and the definition of the measure $\nu$ together
with the the Chinese Remainder Theorem. In fact, for technical
reasons, it is better to work with a different semi-algebra that also 
generates the Borel $\sigma$-algebra on $\mathbb X(V)$~\cite{H}. 

Vanishing measure-theoretical entropy (relative to $\nu$) was shown
in~\cite[Thm.~4]{PH}, which is in line with the results
of~\cite{BLR}. As a consequence of part (a), the individual
diffraction measure of $V$ according to Theorem~\ref{thdiff} coincides
with the diffraction measure of the system $(\mathbb
X(V),\varLambda,\nu)$ in the sense of~\cite{BL}. Then, pure point
diffraction means pure point dynamical spectrum~\cite[Thm. 7]{BL},
and the latter is the group generated by the Fourier--Bohr spectrum;
compare~\cite[Thm. 8]{BL} and~\cite[Prop. 17]{BLvE}. Since the intensity
formula of Theorem~\ref{thdiff} shows that there are no extinctions,
the Fourier--Bohr spectrum here is itself a group, which completes
part (b).

The Kronecker system can now be read off from the model set
description, which provides the compact Abelian group. For the cases
$k=1$ and $d\ge 2$ as well as $k\ge 2$ and $d=1$, the construction is
given in~\cite{BMP}; see also~\cite[Ch. 5a]{Sing} for an alternative
description. The general formalism is developed in~\cite{BLM}, though
the torus parametrisation does not immediately apply. Some extra work
is required here to establish the precise properties of the
homomorphism onto the compact Abelian group.
\end{proof}

Let us mention that our approach is complementary to that
in~\cite{CS}. There, ergodicity and pure point spectrum are consequences of
determining all eigenfunctions, then concluding via $1$ being a simple
eigenvalue and via the basis property of the eigenfunctions. Here, we
establish ergodicity of the measure $\nu$ and afterwards use the equivalence
theorem between pure point dynamical and diffraction spectrum~\cite[Thm. 7]{BL},
hence employing the diffraction measure of $V$ calculated in~\cite{BMP,PH}.

\section*{Acknowledgements}

It is our pleasure to thank Peter Sarnak for valuable discussions. This work was supported by the German Research Foundation (DFG) within
the CRC~701.

\end{document}